\documentclass[12pt]{scrartcl}

\usepackage[utf8]{inputenc}
\usepackage[T1]{fontenc}

\usepackage{amsmath}
\usepackage{hyperref}
\hypersetup{colorlinks=true,unicode}
\usepackage{amsthm}
\usepackage{cleveref}

\usepackage{amssymb,dsfont,mathtools}
\theoremstyle{plain}
\newtheorem{theorem}{Theorem}
\newtheorem*{theorem*}{Theorem}
\newtheorem{proposition}[theorem]{Proposition}
\newtheorem{corollary}[theorem]{Corollary}
\newtheorem{lemma}[theorem]{Lemma}
\newtheorem{assumption}{Assumption}

\theoremstyle{definition}

\newtheorem*{definition*}{Definition}

\theoremstyle{remark}
\newtheorem{remark}[theorem]{Remark}
\newtheorem*{remark*}{Remark}

\usepackage{tikz}
\usepackage{enumerate}

\crefname{figure}{Fig.}{Figs.}
\crefname{theorem}{Theorem}{Theorems}
\crefname{proposition}{Proposition}{Propositions}
\crefname{lemma}{Lemma}{Lemmata}
\crefname{assumption}{Assumption}{Assumptions}

\newcommand{\Bra}[1]{\left\langle #1 \right\rangle }
\newcommand{\abs}[1]{\left\lvert #1 \right\rvert }
\newcommand{\set}[1]{\left\lbrace #1 \right\rbrace }
\newcommand{\norm}[1]{\left\lVert #1 \right\rVert }

\DeclareMathOperator{\divergence}{div}

\DeclareMathOperator{\mass}{\omega}

\DeclareMathOperator{\pr}{Pr}
\DeclareMathOperator{\supp}{supp}
\DeclareMathOperator{\dd}{d\hspace{-2px}}
\newcommand{\R}{\mathds{R}}
\newcommand{\eps}{\varepsilon}

\title{Heat-content and diffusive leakage from material sets in the low-diffusivity limit\thanks{This work is supported by the Priority Programme SPP 1881 \emph{Turbulent Superstructures} of the German Research Foundation.}
}

\author{
Nathanael Schilling\thanks{Department of Mathematics, Technical University of Munich, Germany, \href{mailto:schillna@ma.tum.de}{schillna@ma.tum.de}}
\and
Daniel Karrasch\thanks{Department of Mathematics, Technical University of Munich, Germany, \href{mailto:karrasch@ma.tum.de}{karrasch@ma.tum.de}}
\and
Oliver Junge\thanks{Department of Mathematics, Technical University of Munich, Germany, \href{mailto:oj@tum.de}{oj@tum.de}}
}

\begin{document}

\maketitle

\begin{abstract}
We generalize leading-order asymptotics of a form of the \emph{heat content
of a submanifold} (van den Berg \& Gilkey 2015) to the setting of time-dependent diffusion
processes in the limit of vanishing diffusivity. Such diffusion
processes arise naturally when advection-diffusion processes are viewed in
Lagrangian coordinates. We prove that as diffusivity $\eps$ goes to zero,
the diffusive transport out of a material set $S$ under the
time-dependent, mass-preserving advection-diffusion equation with initial
condition given by the characteristic function $\mathds{1}_S$, 
is $\sqrt{\eps/\pi}\dd\overline{A}(\partial S) + o(\sqrt{\eps})$.
The surface measure $\dd \overline A$ is that of the so-called
\emph{geometry of mixing}, as introduced in (Karrasch \& Keller 2020).
We apply our result to the characterisation of coherent structures in time-dependent dynamical systems.
\end{abstract}


\section{Motivation}

Consider the advection-diffusion process of a passive scalar $u$ by a
sufficiently regular, possibly time-dependent, volume-preserving vector field $V$ as described by the advection-diffusion equation
\begin{equation} \label{eq:ADE}
\partial_t u = - \divergence(uV) + \eps \Delta u\,,
\end{equation}
and some initial condition $u(0,\cdot) = u_0$.  Let $\Phi_0^t$ denote the
flow map (from time $0$ to time $t$) induced by $V$. For $\eps = 0$, there
is only advection and the time-$t$ solution operator mapping $u(0,\cdot)$
to $u(t,\cdot)$ under \cref{eq:ADE} is given by a coordinate change by the
flow map of $V$, i.e., $u(t,\Phi_0^t(x)) = u_0(x)$.
The coordinates induced by the flow map $\Phi_0^t$ are well known as \emph{Lagrangian}
coordinates. We refer to flow-invariant space-time sets as \emph{material sets}. 

For any non-negative $\eps$, we are interested in the leakage of $u$ from a full-dimensional material
set $S$ with smooth boundary over the time interval $[0,t]$ under \cref{eq:ADE}. Let us denote this
\emph{material outflow} by
\[
T_0^t(S,u_0,\eps)\coloneqq\int_S u_0\,\dd x-\int_{\Phi_0^t(S)}u(t,x)\,\dd x\,.
\]
In the advection-only case, flow-invariance directly implies
\begin{equation}\label{eq:flowinvariance}
T_0^t(S,u_0,0) = 0\,,
\end{equation}
regardless of the initial condition $u_0$ and set $S$. In simple terms, no mass can leak
out of a material set if there is no diffusion. 

For $\eps > 0$, however, the situation is different: In general, $T_0^t(S,u_0,\eps)$ does
not vanish, and the asymptotics of $T_1^t(S,u_0,\eps)$ as $\eps \to 0$ are nontrivial and
of both scientific and practical interest. In \cite{Haller2018}, leading-order
asymptotics of $T_0^t(S,u_0,\eps)$ for smooth $u_0$ compactly supported in the interior
of the domain were derived, and in \cite{Karrasch2020a} they were additionally studied
from a geometric point of view.
In this work, we further expand the theory towards the natural case $u_0 = \mathds{1}_S$.

Let $\tilde u$ denote $u$ in Lagrangian coordinates, i.e., $\tilde u(t,\cdot) = u(t,\cdot)\circ \Phi_0^t$. Then \cref{eq:ADE} reads as
\begin{equation}\label{eq:LADE}
\partial_t \tilde u = \eps \Delta_t \tilde u\,,
\end{equation}
where $\Delta_t$ is the differential geometrical pullback of the Laplace
operator by $\Phi_0^t$; see, for instance, \cite{Press1981,Thiffeault2003,Karrasch2020c}.
With a common, slight abuse of notation, we will omit the tilde in \cref{eq:LADE}
henceforth as we work in Lagrangian coordinates exclusively. Here,
\[
T_0^t(S,\mathds{1}_S, \eps) = \int_S \dd x - \int_S u(t,x)\, \dd x = \int_{M \setminus S} u(t,x) \,\dd x\,,
\]
since $\int_{M} u(t,x) \,\dd x = \int_M u_0(x) \,\dd x = \int_S \dd x$ for all $t\in[0,1]$ by mass preservation.
If \cref{eq:LADE} were the classical autonomous heat equation, then the leading-order coefficient (of order $\sqrt \eps$) in $T_0^1(S, \mathds{1}_S, \eps)$ is proportional to the surface area of $\partial S$; see \cite{Vandenberg2015}.
For a generalization to the nonautonomous case as in \cref{eq:LADE}, it is
\emph{a priori} unclear whether one should again expect some kind of surface
measure of $\partial S$ in the leading-order coefficient: in the Lagrangian
pullback geometry, $\partial S$ has---in general---a different surface area
at each time instance $t$. Recently, \cite{Karrasch2020c} proposed a
(weighted) geometry---the \emph{geometry of mixing} to be recalled below---which
was developed to specifically analyze advection-diffusion processes on
finite-time intervals. This geometry, which has the mathematical structure
of a weighted (Riemannian) manifold \cite{Karrasch2020c,Karrasch2020a}, admits
an area form $\dd \overline{A}$ about which we show in this work that it,
indeed, determines the leading-order asymptotics of material leakage out of
material sets, namely
\[
T_0^1(S,\mathds{1}_S, \eps) = \sqrt{\frac{\eps}{\pi}} \int_{\partial S} \dd \overline A + o\left(\sqrt{\eps}\right)\,.
\]
In our proof, we will work with a generalized form of the time-dependent
Lagrangian heat equation \cref{eq:LADE}, and do not assume that it is
necessarily given as some advection-diffusion equation in Lagrangian
coordinates.

\section{Mathematical setting}

Let $M$ be a smooth compact manifold (possibly with smooth boundary), and
$\mass$ a non-vanishing volume form on $M$. Recall that $\mass$ naturally
defines a divergence operator, acting on vector fields $V \in \Gamma(TM)$, by
$\left(\divergence_{\mass} V\right) \mass = \mathcal L_{V} \mass$. Here,
$\Gamma(TM)$ denotes smooth sections of the tangent bundle and $\mathcal L$
is the Lie derivative. If $(g_t)_{t\in[0,1]}$ is a smoothly-varying
one-parameter family of Riemannian metrics on $M$, a \emph{weighted} Laplace
operator (cf.~\cite{Grigoryan2009}), acting on smooth functions $f \in C^\infty(M)$, is defined for each
$t$ with the formula 
\[
\Delta_t f \coloneqq  \divergence_{\mass} g_t^{-1} \dd f\,.
\]
The notation $g_t^{-1}$ shall be interpreted using the well-known natural identification
of $g_t$ with a vector bundle morphism mapping a tangent vector $v$ to the
cotangent vector $g_t(v,\cdot)$. As $g_t$ is positive definite at each point,
this is in fact a vector bundle \emph{isomorphism} and $g_t^{-1}$ is
well-defined. Indeed, $f \mapsto g_t^{-1} \dd f$ is the \emph{gradient}
induced by the metric $g_t$.

As mentioned earlier, our object of study is the time-dependent heat equation
with diffusivity $\eps > 0$ and initial value $u_0 \in L^2(M,\mass)$,
\begin{equation}\label{eq:lade}
\partial_t u = \eps \Delta_t u\,,\qquad u(0,\cdot) = u_0\,,
\end{equation}
which is a generalization of the classical heat equation on $M$ for which $g_t$ is independent from $t$ and $\mass$ is the Riemannian volume form. 
We will look at \cref{eq:lade} with boundary conditions given by either (i)  $\partial M = \varnothing$, (ii) homogeneous Dirichlet boundary or (iii) homogeneous Neumann boundary. Of course, (i) is a special case of both (ii) and (iii).

\section{The geometry of mixing}

We write $P_t^\eps$ for the the time-$t$ solution operator of \cref{eq:lade},
and denote by $\Bra{\cdot,\cdot}_0$ the $L^2(M,\mass)$ inner product.
Throughout, we will identify the volume form $\mass$ with its induced measure.
A \emph{(time) averaged} version of \cref{eq:lade} describes the
leading-order behaviour of $P_1^\eps$ as $\eps \to 0$.
Indeed, defining\footnote{
The imposed boundary condition type in the definition of the semigroup $\exp(\eps t\overline \Delta)$ corresponds to the one (homogeneous Dirichlet/Neumann) imposed in \cref{eq:lade}, see also \cite{Karrasch2020a}.}
\[
\overline g \coloneqq \left(\int_0^1 g_t^{-1} \dd t \right)^{-1}\,, \quad \overline \Delta \coloneqq  \divergence_{\mass} \overline g^{-1} \dd f\,,\quad\textnormal{and}\quad \overline P_t^\eps \coloneqq \exp(\eps t \overline \Delta)\,,
\]
it is true \cite{Karrasch2020a}, see also \cite{Krol1991,Haller2018}, that for
$u_0 \in C^\infty_c(M)$, i.e., $u_0 \in C^\infty(M)$ with compact support in
$\mathring{M}$,
\begin{equation}\label{eq:averaging1}
\norm{P_1^\eps u_0 - \overline P_1^\eps u_0 }_{L^\infty(M)} = O(\eps^2),\quad \eps \to 0\,.
\end{equation}
The operator $\overline \Delta$ was called the \emph{dynamic Laplacian} in \cite{Froyland2015a},  
and is the natural Laplace operator of the weighted manifold
$(M,\overline g, \mass)$. This weighted manifold was coined \emph{geometry of
mixing} in \cite{Karrasch2020c}. On the surface $\partial S$ oriented by
the $\overline g$-unit outer normal vector field $\nu$, the geometry of
mixing has a natural area form given by $\dd \overline A(\cdot)\coloneqq \mass(\nu,\cdot)$; see \cite{Karrasch2020c,Karrasch2020a}.

For non-smooth $u_0$, such as $u_0 = \mathds{1}_S$, it is not clear whether we can expect a result
like \cref{eq:averaging1}: even in the special case of the
time-\emph{independent} heat equation on the weighted manifold $(M,\overline g,\mass)$ , there are now terms of order $\eps^\frac{1}{2}$ and its powers (which is in stark contrast to the case of smooth initial values). It is known, for example, that 
\begin{equation}\label{eq:simplercase}
\Bra{\overline P^\eps_1 \mathds{1}_S, \mathds{1}_{M \setminus S}} = \sqrt{\frac{\eps}{\pi}}\int_{\partial S} \dd \overline{A} + o(\sqrt{\eps})\,,\quad \eps \to 0\,;
\end{equation}
see \cite{Vandenberg2015,Schilling2020}. In this paper, we show that
\cref{eq:simplercase} remains true if $\overline P^\eps_1$ is replaced by
$P^\eps_1$, i.e.\ the following theorem which we prove in \cref{sec:sec6}.

\begin{theorem}
\label{thm:stochastictheorem}
Let $S$ be a compact, full-dimensional submanifold of $M$ with smooth boundary, contained in the interior of $M$, and let $S^c \coloneqq M \setminus M$. Then
\begin{equation}\label{eq:toprovesasection}
\Bra{P_1^\eps \mathds{1}_S, \mathds{1}_{S^c}}_0  = \sqrt{\frac{\eps}{\pi}}\int_{\partial S} \dd\overline A + o(\sqrt \eps)\,, \quad\eps\to 0\,.
\end{equation}
\end{theorem}

\section{Eulerian coherent pairs and Lagrangian coherent sets}

In \cite{Froyland2014}, the following concept of \emph{coherence} has been introduced.
Consider two spatial sets $S$ (at time 0) and $S'$ (at time 1), $L_\eps$ a small
perturbation of the \emph{transfer operator}, i.e., the solution operator for
\cref{eq:ADE} with $\eps=0$, where the perturbation strength scales with
$\eps>0$. Then \cite{Froyland2014} proposed a \emph{coherence ratio}
\begin{equation}\label{eq:coherence}
    \rho_\eps(S,S') \coloneqq \frac{\Bra{L_\eps\mathds{1}_S, \mathds{1}_{S'}}_0}{\mass(S) } + \frac{\Bra{L_\eps\mathds{1}_{M\setminus S}, \mathds{1}_{M \setminus S'}}_0}{\mass(M \setminus S')}\,,
\end{equation}
as a measure of \emph{coherence} of the pair $(S,S')$. Verbally, this measures how much
of $S$ is carried to $S'$ and how much of $M\setminus S$ is carried to $M\setminus S'$ by
the ``perturbed flow''. In yet other words, \emph{coherent pairs} (see also \cite{Banisch2017}) $S$ and $S'$ are pairs
of spatial sets such that there is little leakage under the action of $L_\eps$.

Of course, choosing $S' = \Phi_0^1(S)$ results in no leakage (or, equivalently, coherence
ratio equal to 1) in the non-diffusive case, notably for any choice of $S$, see
\cref{eq:flowinvariance}. While, in that limit case, the problem of seeking ``maximally
coherent pairs'' becomes meaningless, one would expect that $S' = \Phi_0^1(S)$ is the
right condition to perturb from when bringing weak diffusion into consideration.

We thus define the \emph{Lagrangian coherence ratio} as 
\[
\tilde \rho_\eps(S) \coloneqq \rho_\eps(S,\Phi_0^1(S))\,.
\]
Seemingly trivial, this has conceptually deep implications. First, it removes one degree
of freedom, the choice of $S'$. As a consequence, it changes the focus from Eulerian
coherent pairs (of sets) to individual Lagrangian coherent sets. Moreover, it is clear
that, for given $S$, the Lagrangian coherence ratio depends only on the type and the
strength of the perturbation of the transfer operator. One implementation of a
perturbation, as done in \cite{Froyland2014}, is to convolve densities both before and
after the purely advective transport with an explicitly-defined kernel, whose support is
bounded by $\eps$ away from 0.
Another popular approach is to omit any explicit perturbation, and rely on ``numerical
diffusion'' (e.g., via box discretizations) instead; see \cite{Froyland2010a}.
The choice $L_\eps = P^\eps_1$, i.e., the solution operator to the
Lagrangian advection-diffusion equation \cref{eq:LADE} was suggested in
\cite{Karrasch2020c}---see \cite{Denner2016} for the analogous Eulerian approach---as a
physically natural perturbation candidate, that can also be given a stochastic
interpretation. With this definition of $L_\eps$, we can work with indicator
functions directly when maximizing coherence measures like \cref{eq:coherence}, instead
of applying a two-step relaxation procedure as is sometimes done; see \cite{Huisinga2006,Froyland2014,Denner2017}.

By \cref{thm:stochastictheorem}, if $\partial S$ is smooth and $\partial M = \emptyset$ (or with homogeneous Neumann boundary) , then as $\eps \rightarrow 0$,
\begin{align*}
\rho_\eps(S,\Phi_0^1(S)) = \frac{\mass(S) - \sqrt{\frac{\eps}{\pi}}\int_{\partial S} \dd \overline A }{\mass(S)} + \frac{\mass(M \setminus S) - \sqrt{\frac{\eps}{\pi}} \int_{\partial S} \dd \overline A}{\mass(M \setminus S)}  + o(\sqrt \eps)\,.
\end{align*}
In other words, if we fix $\omega(S)$, the coherence ratio depends in leading order as
$\eps \rightarrow 0$ only on (a constant times) the area of $\partial S$ in the
geometry of mixing. Smooth local minimizers of the area functional in a (weighted)
manifold with respect to volume-preserving variations are well-known to be surfaces of
constant generalized mean curvature; see \cite[Sect.~9.4E]{Gromov2003}. The above
considerations hence suggest that sets bounded by such a minimizing surface be viewed as
\emph{Lagrangian coherent sets} in the low-diffusivity limit. This connection between the
concept of coherent sets and that of the (generalized) isoperimetric problem is closely
related to the connection described in \cite{Froyland2015a}. At the same time, it has
close ties to the studies of diffusive transport across material surfaces performed in
\cite{Karrasch2020c,Haller2018,Haller2020}.

\section{Proof of the main theorem}
\label{sec:sec6}

\subsection{Overview}

Our proof consists of a reduction of \cref{eq:toprovesasection} to the
time-independent setting so that we can apply \cref{eq:simplercase}.
In a first step, we perform this reduction for the case $M=\R^n$
in \cref{sec:step1} using stochastic methods. This avoids technical
complications arising from dealing with manifold-valued stochastic processes.
We then treat the general case where $M$ is an
arbitrary compact manifold in a second step
(\cref{sec:reductionstep}).

The structure of the first step is sketched in \cref{fig:proof_structure}.  
On the top right hand side we depict the averaged, i.e., time-independent,
advection-diffusion equation for which we know (\cref{eq:simplercase})
the asymptotic behaviour of $\Bra{P_1^\eps \mathds{1}_S, \mathds{1}_{S^c}}_0$
as $\eps \to 0$. On the top left there is the \emph{time-dependent}
advection-diffusion equation which is the subject of \cref{thm:stochastictheorem}.
Each arrow represents a reduction or approximation step in the proof:
\begin{enumerate}[(i)]
\item 
The upper two arrows (blue) connect a stochastic differential equation (SDE) to its
Kolmogorov backward PDE above it; see \cref{sec:kbe}.
\item Central arrows (olive): Each SDE is approximated by another SDE, inheriting the
leading-order asymptotics we are interested in, see \cref{sec:stochana}.
The use of this kind of SDE approximation to obtain PDE approximations is well known in the literature, see, e.g., \cite[Section 2.3]{Freidlin2012}.
\item The lower arrow (black) highlights the fact that $Y_t^\eps$ and $\overline Y_t^\eps$ have the same law, and, as a consequence, they share the same the leading-order asymptotics of interest.
\end{enumerate}
The reduction as a whole may be conceptualised as going along the arrows from the top right of \cref{fig:proof_structure} to the top left.

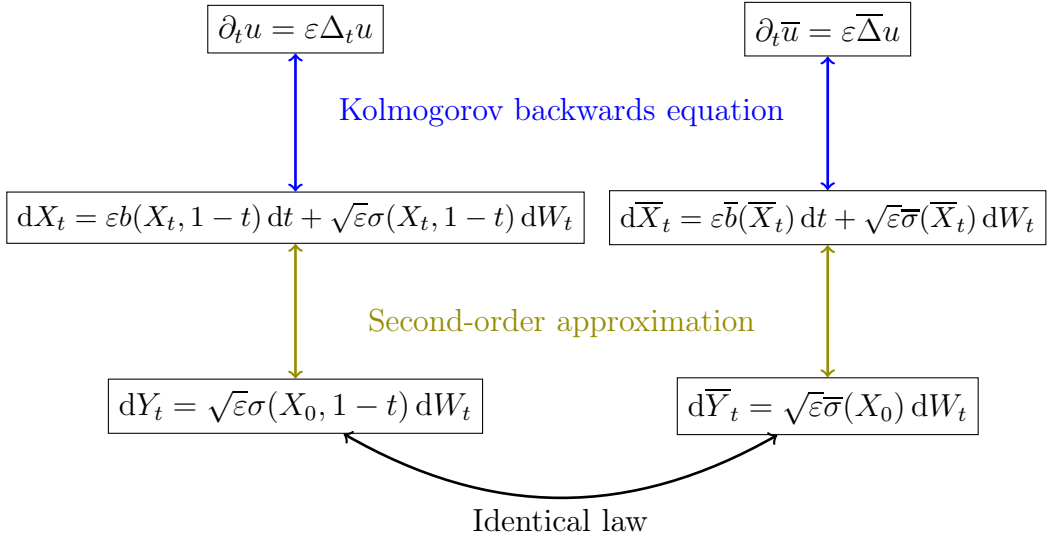
\begin{figure}\begin{center}
\begin{tikzpicture}[node distance=0.15\textwidth]
\node[draw] (A) at (0,0) {
$\partial_t u = \eps \Delta_t u$};

\node[draw] (B) at (7.0,0.0) {
$\partial_t \overline u = \eps \overline \Delta u $};

\node[draw] (C) at (0.0,-2.5) {\small
$\dd X_t = \eps b(X_t,1-t)\dd t + \sqrt \eps \sigma(X_t,1-t)\dd W_t $};

\node[draw] (D) at (0.0,-5.0) {
$\dd Y_t = \sqrt \eps \sigma(X_0,1-t)\dd W_t $};

\node[text width=0.5\textwidth,text centered,color=blue] (G) at (3.5,-1.1) {
Kolmogorov backwards equation};

\node[text width=0.5\textwidth,text centered,color=olive] (G) at (3.5,-3.9) {
Second-order approximation};

\node[draw] (E) at (7.0,-2.5) {\small
$\dd \overline X_t = \eps \overline b(\overline X_t)\dd t + \sqrt{\eps}\overline\sigma(\overline X_t)\dd W_t $};

\node[draw] (F) at (7.0,-5.0) {
$\dd \overline Y_t = \sqrt \eps \overline\sigma(X_0)\dd W_t $};

\draw [<->, color=blue, line width=1pt] (A) edge (C);
\draw [<->, color=olive, line width=1pt] (C) edge (D);

\draw [<->, color=blue, line width=1pt] (B) edge (E);
\draw [<->, color=olive, line width=1pt] (E) edge (F);

\draw [<->, color=black, line width=1pt] (D) edge[bend right] node [below] {Identical law} (F);
\end{tikzpicture}
\end{center}
\caption{Schematic visualization of the structure of the proof of \cref{thm:stochastictheorem} on $\R^n$.}
\label{fig:proof_structure}
\end{figure}

\subsubsection*{Technical issues caused by non-compactness}

As $\R^n$ is not compact, it may be that $\mathds{1}_{\R^n} \notin L^2(\R^n,\mass)$.
This means that the $\Bra{\cdot,\cdot}_0$ notation
appearing in \cref{eq:toprovesasection} must be clarified:
we abuse notation by writing $\Bra{f,g}_0 \coloneqq \int_M f(x) g(x)\mass$ whenever $fg \in L^1(\R^n,\mass)$.
Similar issues also play a role in the non-compact case. Since in
\cref{thm:stochastictheorem} we assume the set $S$ to be compact anyway, in
order to avoid unncessary technical complications, we state the following simplifiying

\begin{assumption}\label{assumption}
There exists a bounded set $B$ (containing $S$ in its interior) so that both $g_t$ and $\mass$ are equal to the Euclidean metric and its volume form respectively outside of $B$ for all $t \in [0,1]$.
\end{assumption}

\subsection{Step 1: The case \texorpdfstring{$M=\R^n$}{M=ℝ\^n}}
\label{sec:step1}

On $M=\R^n$, the initial value problem \cref{eq:lade} takes the form
\begin{equation}\label{ladecoords}
\partial_t u_\eps = \eps\left(\sum_{i=1}^n b_i \partial_i u_\eps + \frac12 \sum_{i,j=1}^n a_{i,j} \partial_{ij} u_\eps \right), \qquad u_\eps(0,\cdot) = u_0\,.
\end{equation}
Here, the space-time-dependent, real-valued functions $a_{ij}$ and $b_i$ depend on the metrics $(g_t)_{t\in[0,1]} $ and the volume form $\mass$.
There are no coefficients of lower order because $\Delta_t \mathds{1}_{\R^n} = 0$ for all $t \in [0,1]$. \Cref{assumption}
yields that on the complement of $B$, $a_{i,j} = \delta_{ij}$ and $b_i = 0$ in Cartesian coordinates.
We have collected some results from the literature on parabolic PDEs in \cref{sec:parabolic} adapted to our setting which we will use in the sequel.

\subsubsection{The Kolmogorov backwards equation}
\label{sec:kbe}

The time-$1$ solution operator of \cref{ladecoords}
is closely linked to the  stochastic process governed by the SDE
\begin{equation}\label{eq:sdemain}
\dd X_t^\eps = \eps b(1-t,X_t^\eps)\dd t + \sqrt{\eps} \sigma(1-t,X_t^\eps)\, \dd W_t\,,
\end{equation}
with $\left(\sigma(t,x)\sigma^\top(t,x)\right)_{ij} = a_{ij}(t,x)$ and initial value $X_{t_0}^\eps = X_{t_0}$. 
It is well known that for a given $n$-dimensional Brownian motion $\left(W_t\right)_{t \in [0,1]}$,
for $t_0 \in [0,1]$ a unique strong solution to \cref{eq:sdemain}, starting at time $t_0$, exists provided that $X_{t_0}$ is independent of $(W_t)_{t\in[t_0,1]}$ and that $b$ and $\sigma$ satisfy Lipschitz and growth conditions (cf.~also \cref{sec:stochana}).
A direct consequence of smoothness and \cref{assumption} is that the Lipschitz and growth conditions are satisfied, as it is well known that $\sigma$ may be chosen to be smooth. To explicitly include the dependence of the process $(X_t^\eps)_{t\in[t_0,1]}$ on the random variable $X_{t_0}$,  we will write $E_{t_0,x}[\cdot]$ (given $x \in \R^n$) for the expected value under the assumption that $X_{t_0} = x$ almost surely; in this case $X_0$ has law given by the Dirac delta measure centered at $x$.
The \emph{Kolmogorov backwards-equation} associated to \cref{eq:sdemain} is a partial differential
equation (PDE) for the function
\[
w_\eps(t,x) = E_{t,x}[u_0(X_1^\eps)]\,,
\]
provided that $u_0$ is sufficiently smooth, see \cite[Theorem 6.1]{Friedman1975}.
This PDE reads
\[
\partial_t w_\eps(t,x) = -\eps\left(\sum_{i=1}^n b_i(1-t,x)\partial_i w_\eps(t,x) + \frac12 \sum_{i,j=1}^n a_{ij}(1-t,x) \partial_{ij} w_\eps(t,x) \right),
\]
and moreover $w_\eps(t,x) \to u_0(x)$  as $t \to 1$. Thus,
$u_\eps(t,x) = w_\eps(1-t,x)$, as both sides satisfy \cref{ladecoords} and
solutions to parabolic PDEs are unique. As a consequence,
\begin{equation}\label{eq:stochkoopman}
(P^\eps_1 u_0)(x) = u_\eps(1,x) = w_\eps(0,x) = E_{0,x}[u_0(X_1^\eps)]\,.
\end{equation}
This equation provides a probabilistic interpretation of the time-$1$ solution
operator of \cref{ladecoords} in terms of the SDE defined by \cref{eq:sdemain}. 
\subsubsection{Probabilistic interpretation of the heat content in a manifold}
In \cref{eq:stochkoopman} we assume the process $X_t$ to start at the
constant $x$ almost surely, i.e., we choose the initial value $X_0$ to have
law equal to the point measure at $x$.  We may, however, also treat the case
in which the initial value $X_0$ of \cref{eq:sdemain} is no longer a constant
random variable.

Let $h\colon M \to \R_{\geq 0}$ be a measurable function so that $h\mass$ is a probability measure.
We denote by $E_{h}[\cdot]$ the expected value in a probability 
space where $X_0$ has law $h \mass$ independent of the Brownian motion $(W_t)_{t \in [0,1]}$.
One may verify that\footnote{To see this, we first observe that the Markov property of SDEs \cite[Thm.~9.2.3]{Arnold1974} yields a time-$1$ transition function $p_\eps$ satisfying $p_\eps(x,A) = E_{x,0}[\mathds{1}_A(X_1)]$
for $x \in \R^n$ and measurable $A \subset \R^n$.   The definition of the inner product $\Bra{\cdot,\cdot}_0$ yields
\[
\Bra{x \mapsto E_{0,x}[\mathds{1}_A(X_1))], h}_0 = \int_{\R^n} p_\eps(\cdot,A)h(\cdot)\mass =  
E_{h}[\mathds{1}_A(X_1)].
\]}
\[
\Bra{x \mapsto E_{0,x}[u_0(X_1))], h}_0 = E_{h}[u_0(X_1)]
\]
holds in the case $u_0 = \mathds{1}_A$, the extension to all $u_0 \in L^\infty(\R^n)$
follows from linearity and monotone convergence.
Using \cref{eq:stochkoopman} and \cref{remark:boundedConv}, it follows that
\begin{equation}\label{eq:fubinitricka}
\Bra{P^\eps u_0, h }_0 = \Bra{x\mapsto E_{x, 0}[u_0(X_1)],h }_0 =  E_{h}[u_0(X_1)]\,.
\end{equation}
We summarize that \cref{eq:fubinitricka} proves a \emph{probabilistic}
interpretation of inner products of the form $\Bra{P^\eps_1 u_0, h}_0$
provided that (i) $u_0 \in L^\infty(\R^n)$, and (ii) $h \in L^1(\R^n,\mass)$ is nonnegative.
The inner product appearing in \cref{eq:toprovesasection} is not of the form
just discussed as $\mathds{1}_{S^c}$ is not in general in $L^1(\R^n, \mass)$. 
Observe, however, that for compact $S$, 
\[
\Bra{P^\eps \mathds{1}_S, \mathds{1}_{S^c}}_0 = \Bra{\mathds{1}_S, P^\eps \mathds{1}_{S^c}}_0\,,
\]
which is proven in \cref{lemma:kindofselfadjoint} in \cref{sec:misc}.
As a consequence of this and \cref{eq:fubinitricka},
the left hand side of \cref{eq:toprovesasection} may be re-written as
\begin{equation}\label{eq:toprovesasection2}
\Bra{P^\eps \mathds{1}_S, \mathds{1}_{S^c}}_0 = E_{\mathds{1}_S}[\mathds{1}_{S^c}(X^\eps_1))]\,,
\end{equation}
provided that $\mass(S) = 1$, which can be assumed without loss of generality.
\subsubsection{Probabilistic interpretation of heat content in the averaged setting}
The steps above correspond to the left blue arrow in \cref{fig:proof_structure}. The right blue arrow corresponds to repeating the same construction for the \emph{averaged} equation $\partial_t \overline u = \eps \overline \Delta \overline u$.
Here, the PDE of $\overline u$ is given in coordinates by
\[
\partial_t \overline u_\eps = \eps\left(\sum_{i=1}^n \overline b_i \partial_i \overline u_\eps + \frac12\sum_{i,j=1}^n \overline a_{i,j}\partial_{ij}\overline u_\eps\right),
\]
with $\overline b_i(x) = \int_0^1 b(t,x)\dd t$ and $\overline a_{i,j}(x) = \int_0^1 a_{i,j}(t,x)\dd t$. The associated stochastic process is defined by the SDE
\begin{equation}\label{eq:SDE_overline}
\dd \overline X_t^\eps = \eps \overline b(\overline X_t^\eps)\dd t + \sqrt \eps \overline \sigma(\overline X_t^\eps)\dd W_t\,,
\end{equation}
with $\overline \sigma \overline \sigma^\top = \overline a$. 
Given initial value $\overline X_0^\eps = X_0$, we see that analogously to \cref{eq:fubinitricka}, 
\[
\Bra{\overline P^\eps u_0 , h}_0 = E_{h}[u_0(\overline X_1^\eps)]
\]
holds when $u_0 \in L^\infty(M,\mass)$ and $h\mass$ is a probability measure.
Our aim is now to show that
\begin{equation}\label{eq:aimkol}
E_{\mathds1_S}[\mathds{1}_{S^c}(X_1)] =
E_{\mathds1_S}[\mathds{1}_{S^c}(\overline X_1)] + o(\sqrt{\eps})\,,\quad \eps \to 0\,,
\end{equation}
corresponding to $h=\mathds{1}_S$. In fact, we generalize to positive
\[
h \in C^\infty_S(M) \coloneqq \set{f\mathds{1}_S;~f \in C^\infty(M)},
\]
so that $h\mass$ is a probability measure, and will look at the quantity $E_{h}[\mathds{1}_{S^c}(X_1^\eps)] = \Bra{P^\eps \mathds{1}_{S^c}, h}_0$ with the aim of showing
\begin{equation}\label{eq:aimkol3}
E_{h}[\mathds{1}_{S^c}(X_1)] =
E_{h}[\mathds{1}_{S^c}(\overline X_1)] + o(\sqrt{\eps})\,,\quad\eps \to 0\,.
\end{equation}
Writing $h = f\mathds{1}_S$, we know from \cite{Vandenberg2015,Schilling2020}, that 
\begin{equation}\label{eq:finalstep1}
E_{h}[\mathds{1}_{S^c}(\overline X_1)] = \Bra{\overline P^\eps \mathds{1}_{S^c}, f \mathds{1}_S}_0 = \sqrt\frac{ \eps}{ \pi} \int_{\partial S} f \dd \overline A + o(\sqrt \eps)\,,\quad\eps \to 0\,,
\end{equation}
which yields the asymptotic behaviour of the right hand side of \cref{eq:aimkol3}.
Our aim in the next steps will be to prove \cref{eq:aimkol3}.

\subsubsection{Approximation of stochastic processes}
\label{sec:stochana}

We continue with the middle (green) arrows in \cref{fig:proof_structure},
starting with the left one.
Here, we will construct a family of stochastic processes $(Y_t^\eps)_{t \in [0,1]}$ so that
\begin{equation}\label{eq:approxstoch}
E\left[\abs{X_t^\eps - Y_t^\eps}^2\right] \leq K\eps^2\,
\end{equation}
for some $K > 0$ and all $t\in[0,1]$ for sufficiently small $\eps$.
In light of the arguments around \cref{eq:aimkol3}, we will use this approximation to show that:
\begin{proposition}\label{prop:approximationprop}
If $(Y_t^\eps)_{t\in[0,1]}$ satisfies \cref{eq:approxstoch}, then for $h \in C^\infty_S(M)$, 
\begin{align}\label{eq:aimkol2}
E_{h}[\mathds{1}_{S^c}(X_1^\eps)]  = 
E_{h}[\mathds{1}_{S^c}(Y_1^\eps)] + o(\sqrt \eps).
\end{align}
Analogously, corresponding to the the right hand side of \cref{fig:proof_structure}, if the family of processes $(\overline Y_t^\eps)_{t\in[0,1]}$ satisfies an inequality like \cref{eq:approxstoch} but with $\overline X_t^\eps$ in place of $X_t^\eps$, then
\begin{align}\label{eq:aimkol4}
E_{h}[\mathds{1}_{S^c}(\overline X_1^\eps)]  = 
E_{h}[\mathds{1}_{S^c}(\overline Y_1^\eps)] + o(\sqrt \eps)\,.
\end{align}
\end{proposition}
The processes $Y_1$ and $\overline Y_1$ will have the same law (this is the bottom arrow in \cref{fig:proof_structure}), after we have proven this we may conclude that
$E_{h}[\mathds{1}_{S^c}(Y_1^\eps)] = 
E_{h}[\mathds{1}_{S^c}(\overline Y_1^\eps)]$, which yields \cref{eq:aimkol3},
which together with \cref{eq:aimkol2,eq:aimkol4} shows (here $h = \mathds{1}_S f$), that
\[
E_{h}[\mathds{1}_{S^c}(X_1^\eps)] 
 = E_{h}[\mathds{1}_{S^c}(\overline X_1^\eps)] + o(\sqrt\eps) \\
 = \sqrt\frac{\eps}{\pi} \int_{\partial S} f \,\dd \overline A + o(\sqrt\eps)\,.
\]

Before proving \cref{prop:approximationprop}, we will state a lemma needed in the proof. 
Let $A$ be a (Borel) measurable subset of $\R^n$. We denote by
$d(x,A)=\inf_{a\in A} \abs{x-a}$ the Euclidean distance between a point
$x\in\R^n$ and the set $A$.  Let further $A_\delta \coloneqq \set{x \in \R^n; ~d(x,\partial A) \leq \delta}$ be the $\delta$-neighborhood of the boundary of $A$.

\begin{lemma}\label{lemma:asymptoticvalidity}
Let $(\Omega,\mathcal A,\mathcal \pr)$ be some probability space and let $E[X]$ denote the expectation of some random variable $X$ on $\Omega$ with respect to $\pr$. For $\eps \in [0,1]$, let $A^\eps$ and $B^\eps$ be $(\R^n,\mathcal B)$-valued random variables with $E[\abs{A^\eps - B^\eps}^2] \leq C_0^2\eps^2$ for some $C_0 > 0$.  
Let $R\in\mathcal B$ and assume that $\mathcal P(A^\eps \in R_\delta) \leq C_1 \delta$ for sufficiently small $\delta > 0$ and some $C_1>0$.  Then,
\[
\abs{E[\mathds{1}_R (A^\eps)] - E[\mathds{1}_R(B^\eps)]} = o(\sqrt \eps)\,\quad \eps \to 0.
\]
\end{lemma}

\begin{proof}
The proof is given in \cref{sec:misc}, and is essentially an application of the Markov inequality.
\end{proof}
\begin{proof}[Proof of \cref{prop:approximationprop}]
We will apply \cref{lemma:asymptoticvalidity} twice with $R = S^c$.
For the first application, with $A^\eps = X_1^\eps$ (corresponding to \cref{eq:aimkol2}), we will need to check that $\pr[X_1 \in (S^c)_\delta] \leq C_1\delta$ for some constant $C_1 > 0$. To see this is indeed the case, observe that $(S^c)_\delta = S^\delta$, and furthermore
$\pr(X_1^\eps \in S^\delta) = E[\mathds{1}_{S_\delta}(X_1^\eps)]$. In the case that $X_0$ has law $f \mathds{1}_S\mass$, this is equal to $E_{f\mathds{1}_S }[\mathds{1}_{S_\delta}(X_1^\eps)]$. Thus, if $X_0$ has law $f \mathds{1}_S\mass$ we have
\begin{align*}
\mathcal P(X_1 \in (S^c)_\delta) 
& = \Bra{P^\eps \mathds{1}_{S_\delta}, f}_0 \\
& \leq \norm{P^\eps \mathds{1}_{S_\delta}}_{L^1(M,\mass)} \norm{f}_\infty 
  \tag{by Hölder's inequality}\\
& \leq \mass(S_\delta) \norm{f}_\infty 
  \tag{by mass preservation of $P^\eps$} \\
& \leq C_1\delta \,,
\end{align*}
for some $C_1 > 0$, proving the claim. The proof required for the second application
(with $A = \overline X_1$, i.e., \cref{eq:aimkol4}) that $\mathcal P(\overline X_1^\eps \in S^\delta) = O(\delta)$ proceeds along the same lines.
\end{proof}
\subsubsection{Approximation by a Gaussian Process}
\label{sec:l2approximations}

We now construct the processes required by \cref{prop:approximationprop} satisfying \cref{eq:approxstoch}. To this end, let $(Y_t^\eps)_{t\in[0,1]}$ be defined by 
\[
\dd Y_t^\eps = \sqrt{\eps}\sigma(1-t,X_0) \, \dd W_t\,,\qquad Y_0^\eps = X_0\,,
\]
where $X_0$ is independent of the Wiener process and bounded. Likewise, let $(\overline Y_t^\eps)_{t\in[0,1]}$ be defined by 
\[
\dd \overline Y_t^\eps  = \sqrt \eps \overline \sigma (X_0)\dd W_t\,, \qquad \overline Y_0^\eps = X_0\,.
\]

\begin{proposition}[\cite{Blagoveshchenskii1961}]
Let $(X_t)_{t\in[0,1]}$ be the stochastic process satisfying \cref{eq:sdemain}. The process $(Y_t^\eps)_{t\in[0,1]}$ approximates $(X_t)_{t\in[0,1]}$ in the sense that
\[
E\left[\abs{X_t^\eps - Y_t^\eps}^2\right] \leq K \eps^2\,, \qquad\text{for all } t\in [0,1].
\]
Similarly, let $(\overline X_t)_{t\in[0,1]}$ by the solution of \cref{eq:SDE_overline}, then $(\overline Y_t^\eps)_{t\in[0,1]}$ approximates $(\overline X_t)_{t\in[0,1]}$ in the sense that
\[
E\left[\abs{\overline X_t^\eps - \overline Y_t^\eps}^2\right] \leq K \eps^2\,, 
\qquad\text{for all } t\in [0,1]\,.
\]
In both cases, $K > 0$ is a constant independent of $\eps$ and $t$.
\end{proposition}

\begin{proof}
This is a special case of the result in \cite{Blagoveshchenskii1961}.
We have adapted the proof of this special case in \cref{sec:stochapprox}.
\end{proof}

The processes $(Y_t^\eps)_{t\in[0,1]}$ and $(\overline Y_t^\eps)_{t\in[0,1]}$ may be thought of as being second-order approximations to the processes $(X_t)_{t\in[0,1]}$ and $(\overline X_t)_{t\in[0,1]}$ respectively.
With \cref{prop:approximationprop}, we conclude that

\begin{proposition}\label{prop:approx}
With $X_1,Y_1,\overline X_1, \overline Y_1$ as defined above and $h \in C^\infty_S(M)$, 
\begin{align*}
\abs{E_{h}[\mathds{1}_{S^c}(X_1)] - E_{h}[\mathds{1}_{S^c}(Y_1)]} &= o(\sqrt \eps)\,,\\
\abs{E_{h} [\mathds{1}_{S^c}(\overline X_1)] - E_{h}[\mathds{1}_{S^c}(\overline Y_1)]} &= o(\sqrt \eps)\,.
\end{align*}
\end{proposition}

While these second-order approximations may differ pointwise, their laws are the same, this is the black arrow in \cref{fig:proof_structure} and the subject of the following lemma.

\begin{lemma} \label{distributionlemma}
The random variables $Y_1^\eps - X_0$ and $\overline Y_1^\eps- X_0$ have the same law, namely that of $\sqrt{\eps}\overline \sigma(X_0)W_1$.
\end{lemma}

\begin{proof}
Recall that $\overline a(x) \coloneqq \int_0^1 a(t,x) \dd t$, $\overline b(x) \coloneqq \int_0^1 b(t,x) \dd t$, and $\overline \sigma \overline \sigma^\top = \overline a$.
As $Y_0 = X_0$, we see that $Y_1^\eps - X_0 = \sqrt{\eps}\int_0^1 \sigma(1-t,X_0)\, \dd W_t$.
If $X_0 = x$ then by \cite[Cor.~4.5.6]{Arnold1974}, the random variable
$Y_1^\eps - X_0$ is a normal random variable with zero mean and covariance
matrix $\eps\int_0^1 \sigma(t,x)\sigma(t,x)^\top\,\dd s = \eps\overline a(x)$, which is
(by the same argument) also the law of $\overline Y^\eps - X_0$. The random
variable $\sqrt{\eps} \overline \sigma(X_0) W_1$ is a normal random variable with
the same mean and covariance matrix, proving the claim for constant $X_0$. The
processes $(Y_t^\eps)_{t\in[0,1]}$ and $(\overline Y_t^\eps)_{t\in[0,1]}$ are not
memoryless as the right hand side depends on the \emph{initial} value of the
process. This can be worked around by suitably augmenting the state space, the
claim of the lemma for nonconstant $X_0$ follows by making use of the Markov
property for SDEs in this augmented state space.
\end{proof}

\begin{corollary}\label{cor:distr}
For $h \in C^\infty_S(M)$, one has that
\[
E_{h}\left[\mathds{1}_{S^c}(Y_1^\eps)\right] = E_{h}\left[\mathds{1}_{S^c}(\overline Y_1^\eps)\right].
\]
\end{corollary}

\begin{proof}
This is a direct result of \cref{distributionlemma}.
\end{proof}

To summarize the reasoning so far: combining \cref{cor:distr} with \cref{prop:approx} yields for $h \in C^\infty_S(M)$, that
\[
E_{h} [\mathds{1}_{S^c}(X_1^\eps)] = E_{h}[\mathds{1}_{S^c}(\overline X_1^\eps)] + o(\sqrt \eps)\,,\quad\eps\to 0\,.
\]
We know that $\Bra{P^\eps \mathds{1}_{S^c}, h}_0 = E_{h} [\mathds{1}_{S^c}(X_1^\eps)]$. Writing $h = \mathds{1}_S f$,  together with \cref{eq:finalstep1}, we may see that
\begin{equation}\label{eq:finalequationstep1}
\Bra{P^\eps \mathds{1}_{S^c}, h}_0 = \sqrt{\frac{\eps}{\pi}} \int_{\partial S} f \dd \overline A + o(\sqrt \eps)\,,\quad\eps\to 0\,.
\end{equation}
With $f \equiv 1$, applying \cref{lemma:kindofselfadjoint} completes the proof of \cref{thm:stochastictheorem}
on $\R^n$ in the setting of \cref{assumption}.

\subsection{Step 2: Restriction to local data and geometry}
\label{sec:reductionstep}

In this section, we write $P^\eps = P^\eps_1$ and $\overline P^\eps = \overline P^\eps_1$.

\subsubsection{Only local data is asymptotically important}

Let $U$ be a compact, full-dimensional submanifold of $M$ with smooth boundary and with $S \subset \mathring{U}$. The inner product
appearing on the left hand side of \cref{eq:toprovesasection} may be written as
\[
\Bra{P^\eps \mathds{1}_S, \mathds{1}_{S^c}}_0 = \Bra{P^\eps \mathds{1}_S, \mathds{1}_{U\setminus S}}_0 + \Bra{P^\eps \mathds{1}_S, \mathds{1}_{U^c}}.
\]
We start by showing that discarding the second term yields an error of $o(\eps)$.

\begin{lemma}\label{lemma:localisation}
Let either $M$ be compact, or $M = \R^n$ (together with \cref{assumption}) and $S,U$ as above.
Let $f \in L^\infty(M,\mass)$ with $\supp(f) \subset \mathring{U}$. Then
\[
\abs{\Bra{P^\eps f, \mathds{1}_{U^c}}_0} = o(\eps)\,,\quad\eps\to 0\,.
\]
\end{lemma}

\begin{proof}
Without loss of generality, assume $f \geq 0$. Pick some
$h \in C^\infty_c(\mathring{M})$ with $f \leq h$ and $\supp(h)\subset \mathring{U}$.  We compute:
\begin{align}
0 \leq \Bra{P^\eps f, \mathds{1}_{U^c}}_0 &\leq \Bra{P^\eps h, \mathds{1}_{U^c}}_0 \label{eq:lemmastep}\\
&= \Bra{P^\eps h, \mathds{1}_M}_0 - \Bra{P^\eps h, \mathds{1}_U }_0\,. \nonumber
\end{align}
If $M$ is compact we are already done at \cref{eq:lemmastep}, as by \cref{thm:approximation}, $P^\eps h = h + \eps \overline \Delta h + O(\eps^2)$ and $\mathds{1}_{U^c} \in L^2(M,\mass)$. 
If $M = \R^n$, we observe that $\Bra{P^\eps h, \mathds{1}_M}_0 = \Bra{h, \mathds{1}_M}_0 = \Bra{h, \mathds{1}_U}_0$. Using \cref{thm:approximation},  $\Bra{P^\eps h, \mathds{1}_U}_0 = \Bra{h + \eps \overline \Delta h + o(\eps), \mathds{1}_U}_0$. Given that $h$ is compactly supported in the interior of $U$, the term $\Bra{\eps \overline \Delta h, \mathds{1}_U}_0$ vanishes by the divergence theorem. We conclude that $\Bra{P^\eps h, \mathds{1}_U}_0 = \Bra{h, \mathds{1}_U}_0 + o(\eps)$  which yields the claim.
\end{proof}

\subsubsection{Only local geometry is asymptotically important}

Similarly, only local {geometry} affects the asymptotic behaviour of
$\Bra{P^\eps \mathds{1}_S, \mathds{1}_{S^c}}_0$.

\begin{lemma}
Let $M$ be either compact (with homogeneous Neumann or Dirichlet boundary conditions) or equal to $\R^n$ (together with
\cref{assumption}). Let $U \subset M$ be a compact, full-dimensional
submanifold. Let $f \in L^\infty(M,\mass)$ with $\supp(f) \subset \mathring{U}$.
Let $\tilde P^\eps$ be defined the same way as $P^\eps$ via \cref{eq:lade} but on $U$ with homogeneous Dirichlet boundary, i.e., $\tilde P^\eps_t$ is the time-$t$ solution operator to the time-dependent diffusion problem \cref{eq:lade} on $U$ with Dirichlet boundary, and $\tilde P^\eps \coloneqq \tilde P^\eps_1$. Then,
\[
\norm{(P^\eps - \tilde P^\eps)f }_{L^\infty(U)} = o(\eps)\,,\quad\eps\to 0\,.
\]
\end{lemma}

\begin{proof}
Let $e_\eps(t) = (P^\eps_t - \tilde P^\eps_t) h$ for a generic positive $h \in C^\infty_c(\mathring{U})$. The function $e_\eps(t)$ satisfies the time-dependent heat equation $\partial_t e_\eps = \eps\Delta_t e_\eps$ on $U$, with nonhomogeneous Dirichlet boundary $e_\eps\bigr|_{\partial_U} = P^\eps_t h\bigr|_{\partial U} \geq 0$. By construction (and reasoning like the following is well known in the literature, cf.~\cite{Grigoryan2006}),
$e_\eps(0,\cdot) = 0$, and the weak maximum principle \cite[Thm.~A.3.1]{Jost2011}  applied to $-e_\eps(t)$ yields that on $U$ and $t\in[0,1]$,
\begin{equation}\label{eq:positivity}
(P^\eps_t - \tilde P^\eps_t)h \geq 0\,.
\end{equation}
By a continuity argument, \cref{eq:positivity} extends to positive $h \in L^\infty (M,\mass)$ with $\supp(h) \subset U$, including $f$.
Going back to the case of a smooth $h$, we may pick $h$ so that $h \geq f$, and for this particular choice we get
\[
0 \leq (P^\eps_t - \tilde P_t^\eps)f \leq (P^\eps_t - \tilde P^\eps_t) h = e_\eps(t)\,,
\]
where both inequalities are a consequence of \cref{eq:positivity}.
We conclude for $e_\eps$, once more with the maximum principle, that
\begin{equation}\label{eq:onboundary}
0 \leq \norm{e_\eps}_{L^\infty([0,1] \times U)} \leq \norm{P_t^\eps h}_{L^\infty([0,1] \times \partial U)}\,.
\end{equation}
With \cref{thm:approximation}, we see that $\norm{P_t^\eps h - \tilde h^\eps}_{L^\infty([0,1] \times M)} = o(\eps)$,
where $\tilde h^\eps(t,\cdot) \coloneqq  h + \eps \int_0^t\, \Delta_s h\,\dd s$. In particular, this $L^\infty$ bound holds also on $[0,1] \times \partial U$ as required in \cref{eq:onboundary} (by
construction, $\tilde u^\eps$ vanishes on $\partial U$).  This shows that
$\norm{e_\eps}_{L^\infty([0,1]\times U)} = o(\eps)$, proving the lemma.
\end{proof}

\subsubsection{Remaining steps}

The statement of \cref{thm:stochastictheorem} is thus reduced to one about
\[
\Bra{\mathds{1}_S, P^\eps \mathds{1}_{S^c}}_0\,,
\]
regardless of what manifold $P^\eps$ is defined on, as long as this manifold
is isometric to the original one on a neighborhood of $S$. By taking a smooth
partition of unity $(f_i)_{i=1}^N$ so that $\sum_{i=1}^N f_i = 1$ on $S$ and
each $f_i$ is supported in a single coordinate chart it is (by linearity)
enough to prove that
\[
\Bra{f_i \mathds{1}_S, P^\eps \mathds{1}_{S^c}}_0 = \sqrt{\frac{\eps}{\pi}}\int_{\partial S} f_i \dd\overline A + o(\sqrt \eps)\,,
\]
for each $i=1,\dots,N$. As each $f_i$ is supported in a single coordinate
chart, we may pick a local isometry into $\R^n$ and prove the
expression there. This is precisely the end result of what was proven in step
1, i.e., \cref{eq:finalequationstep1}, so we are done.

\appendix

\section{Approximation of stochastic process}
\label{sec:stochapprox}

Let $(\Omega, \mathcal F, \mathcal P)$ be a probability space supporting a classical Wiener process $W_t\colon\Omega\to\R^n$ for $t\in [0,1]$. Let $b\colon\R^n\times [0,1]\to\R^n$ and $\sigma\colon\R^n\times [0,1]\to \R^{n\times n}$ be measurable functions. Consider the stochastic initial value problem 
\begin{equation}\label{eq:sde}
\dd X_t^\eps = \eps b(X^\eps_t,t)\, \dd t + \sqrt{\eps}\sigma(X^\eps_t,t)\, \dd W_t\,,\qquad X_0^\eps = X_0\,,
\end{equation}
with initial value $X_0\in L^2(\mathcal P)$. If there is $K>0$ such that  
\[
\abs{b(t,X)-b(Y,t)}+\abs{\sigma(t,X)-\sigma(Y,t)} \leq K \abs{X-Y}\,,
\]
and
\[
\abs{b(t,X)}+\abs{\sigma(t,X)} \leq K \sqrt{1+\abs{X}^2}\,,
\]
for all $t\in [0,1]$, then the initial value problem \cref{eq:sde} has a $\mathcal P$-almost surely unique continuous solution \cite[Theorem (6.2.2)]{Arnold1974}.

The following result is a special case of \cite{Blagoveshchenskii1961}, we
follow the proof there and track the dependence of constants involved on
other values more explicitly. The precise value of $C$, however, may change
from line to line.

\begin{theorem} Let $X_t^\eps$ be the unique solution of \cref{eq:sde} and $Y_t^\eps$  the unique solution of the stochastic initial value problem
\[
\dd Y_t^\eps = \sqrt{\eps}\sigma(X_0,t)\, \dd W_t\,,\qquad Y_0^\eps = X_0\,.
\]
Then there exists $C > 0$ such that for $\eps \leq 1$:
\begin{equation}\label{toprove}
E\left[\sup\limits_{0\leq t \leq 1} \abs{X_t^\eps - Y_t^\eps}^2\right] \leq C \eps^2\,.
\end{equation}
\end{theorem}

\begin{proof}
For $t \in [0,1]$, let
\begin{align*}
\gamma(t,\eps) &\coloneqq \frac{X_t^\eps - Y_t^\eps}{\eps}\,, &
\psi(t,\eps) &\coloneqq E\left(\sup\limits_{0 \leq s \leq t}\abs{\gamma(s,\eps)}^2\right).
\end{align*}
In order to prove \cref{toprove}, it must be shown that that $\psi(1,\eps) \leq C$. By the definition of $\gamma$, 
\[
\eps \gamma(t,\eps) = \underbrace{\eps \int_0^t b(X_s^\eps,s)\,\dd s}_{\coloneqq a_1(t,\eps)} + \underbrace{\sqrt{\eps}\int_0^t \sigma(X_s^\eps,s) - \sigma(X_0,s)\,\dd W_s}_{\coloneqq a_2(t,\eps)}\,,
\]
and furthermore
\begin{align*}
a_1(t,\eps) &= \eps\left( \int_0^t b(X_s^\eps,s) - b(Y_s^\eps,s)\, \dd s + \int_0^t b(Y_s^\eps,s)\, \dd s\right), \\
\abs{a_1(t,\eps)} &\leq \eps K \left( \int_0^t \eps \abs{\gamma(s,\eps)}\, \dd s + \int_0^t \sqrt{1 +  \abs{Y_s^\eps}^2}\, \dd s\right).
\end{align*}
Now, as $t \leq 1$, Jensen's inequality (and $(a+b)^2 \leq 2(a^2 + b^2)$ twice) yields
\begin{align*}
\abs{a_1(t,\eps)}^2 &\leq 2\eps^2K^2\left(\eps^2\left(\int_0^t \abs{\gamma(s,\eps)} \dd s \right)^2 + \left(\int_0^t \sqrt{1 + \abs{Y_s^\eps}^2} \, \dd s\right)^2 \right)\\
&\leq 2 \eps^2 K^2 \left(\eps^2 \int_0^t \abs{\gamma(s,\eps)}^2 \,\dd s + 1 + 2\abs{X_0}^2 +  2\eps \int_0^t \abs{\frac{Y_s^\eps - X_0}{\sqrt \eps}}^2 \, \dd s\right).
\end{align*}
By monotonicity of the Lebesgue integral,
\begin{multline*}
\sup_{0\leq s \leq t}\abs{a_1(s,\eps)}^2 \leq 4 \eps^2 K^2 \left(\eps^2 \int_0^t \sup\limits_{0 \leq u \leq s }\abs{\gamma(u,\eps)}^2\,\dd s  + 1+\right.\\
\left. + \abs{X_0}^2 + \eps \int_0^t\sup\limits_{0 \leq u \leq s} \abs{\frac{Y_u^\eps - X_0}{\sqrt \eps}}^2 \,\dd s\right).
\end{multline*}
Consequently,
\begin{multline}\label{eq:a1bound}
E\left(\sup\limits_{0\leq s \leq t} \abs{a_1(s,\eps)}^2\right) \leq 4 K^2 \eps^2\left(\eps^2 \int_0^t \psi(s,\eps) \, \dd s  + 1+\right.\\
\left. +E(\abs{X_0}^2) 
 +  \eps E\left( \int_0^t \sup\limits_{0 \leq u \leq s} \abs{\frac{Y_u^\eps - X_0}{\sqrt \eps}}^2\,\dd s\right)\right).
\end{multline}
To deal with the $a_2$ term, we use the Itô isometry:
\begin{align*}
E(\abs{a_2(t,\eps)}^2) &= \eps \int_0^t E \left(\abs{\sigma(X_s^\eps,s) - \sigma(X_0,s)}^2\right)\dd s\\
&\leq \eps \int_0^t K^2E\left(\abs{X_s^\eps - X_0}^2\right) \dd s \\
&\leq K^2\eps\int_0^tE\left(\abs{Y_s^\eps - X_0 + \eps \gamma(s,\eps)}^2\right)\dd s\\
&\leq K^2\eps^2\int_0^tE \left(\abs{\frac{Y_s^\eps - X_0}{\sqrt \eps} + \sqrt \eps \gamma(s,\eps)}^2\right) \dd s  \\
&\leq 2 K^2\eps^2 \left(\int_0^t E\left(\abs{\frac{Y_s^\eps - X_0}{\sqrt \eps}}^2\right)\dd s + \eps \int_0^t E(\abs{\gamma(s,\eps)}^2) \dd s   \right)\\
&\leq 2 K^2\eps^2 \left(\int_0^t E\left(\abs{\frac{Y_s^\eps - X_0}{\sqrt \eps}}^2\right)\dd s + \eps\int_0^t \psi(s,\eps)\,\dd s\right).
\end{align*}
As $a_2(t,\eps)$ is a martingale, Doob's maximal inequality for $p=2$ shows that
\begin{align}
E\left(\sup_{0\leq s \leq t}\abs{a_2(s,\eps)}^2\right)  &\leq 4 E(\abs{a_2(t,\eps)}^2)\nonumber \\
&\leq 8 K^2\eps^2 \left(\int_0^t E\left(\abs{\frac{Y_s^\eps - X_0}{\sqrt \eps}}^2\right) \dd s + \eps\int_0^t \psi(s,\eps)\,\dd s\right).\label{eq:a2bound}
\end{align}
Let $Z_t^\eps = \frac{Y_t^\eps - X_0}{\sqrt{\eps}}$. One may readily verify that $Z_t^\eps$ satisfies
\[
Z_t^\eps = \int_0^t \sigma(X_0,t)\,\dd W_t\,,
\]
and hence in particular $Z_t^\eps$ does not depend on $\eps$, so we may write
$Z_t$ without the superscript $\eps$. Moreover, $Z_t$ is an $L^2$-martingale.
Thus, Doob's inequality ensures that $K_1 \coloneqq 4 E[Z_1^2]$ satisfies
$E\left(\sup\limits_{0 \leq t \leq 1}\abs{Z_t}^2\right) \leq K_1$.
Combining \cref{eq:a1bound} with \cref{eq:a2bound} and $(a+b)^2 \leq 2(a^2 + b^2)$, we obtain
\begin{align*}\textstyle
\eps^2 \psi(t,\eps) &\leq
8 K^2 \eps^2\left(\eps^2 \int_0^t \psi(s,\eps) \,\dd s + 1 + E(\abs{X_0}^2) +  \eps E\left( \int_0^t \sup\limits_{0 \leq u \leq s} \abs{Z_u}^2 \, \dd s\right) \right) \\
&\qquad\qquad\qquad+ 16 K^2\eps^2 \left(\int_0^t E(\abs{Z_t}^2)\, \dd s + \eps \int_0^t \psi(s,\eps)\,\dd s   \right) \\
&\leq 16 K^2 \eps^2 \left((\eps + \eps^2)\int_0^t \psi(s,\eps)\,\dd s + 1 + E(\abs{X_0}^2) + (1 + \eps)\int_0^t  E(\sup\limits_{0 \leq u \leq s} \abs{Z_u}^2 )\right).
\end{align*}
Writing $E(\abs{X_0}^2) = K_2 < \infty$ we see that for suitable $D > 0$,
\[
\psi(t,\eps) \leq D \int_0^t \psi(s,\eps)\dd s + 1 + 2K_1 + K_2\,,
\]
assuming $\eps \leq 1$. Grönwall's lemma\footnote{ 
    Observe that $t \mapsto \psi(t,\eps)$ is monotone (and hence measurable) and finite (cf.~\cite[Ch.~5, Cor.~1.2]{Friedman1975}).
    By the monotone convergence theorem, if $t_n \to t$ from below, then $\psi(t_n,\eps) \to \psi(t,\eps)$. In particular, the almost-everywhere (in $t$) bound from the integral form of Grönwall's lemma (see \cite[app B.2]{Evans2010}) holds everywhere.
    },
yields that
$\psi(1,\eps)$
is uniformly bounded, proving the claim.
\end{proof}

\section{Miscellaneous proofs}
\label{sec:misc}

\begin{lemma}
Let $(\Omega,\mathcal A,\pr)$ be some probability space and let $E[X]$ denote the expectation of some random variable $X$ on $\Omega$. For $\eps \in [0,1]$, let $A^\eps$ and $B^\eps$ be $(\R^n,\mathcal B)$-valued random variables with $E\left[\abs{A^\eps - B^\eps}^2\right] \leq C_0^2\eps^2$ for some $C_0 > 0$.  
Let $R\in\mathcal B$ and assume that $\mathcal P(A^\eps \in R^\delta) \leq C_1 \delta$ for sufficiently small $\delta > 0$ and some $C_1>0$.  Then: 
\[
\abs{E[\mathds{1}_R (A^\eps)] - E[\mathds{1}_R(B^\eps)]} = o(\sqrt \eps)\,,\quad \eps \to 0\,.
\]
\end{lemma}

\begin{proof}
Note that for $\mass \in \Omega$, the condition
\[
(A^\eps(\mass) \in R \text{ and } B^\eps(\mass) \in R) \text{ or } (A^\eps(\mass) \notin R \text{ and } B^\eps(\mass) \notin R)
 \]
implies $\mathds{1}_R(A^\eps(\mass)) =  \mathds{1}_R(B^\eps(\mass))$.
We thus may see that for any $\delta > 0$
\begin{align*}
\abs{E[\mathds{1}_R (A^\eps)] - E[\mathds{1}_R(B^\eps)]}
&\leq \pr(A^\eps \in R \,\text{and}\, B^\eps \notin R) \\
&\quad + \pr(A^\eps \notin R \,\text{and}\, B^\eps \in R)\\
&\leq  2 \pr(A^\eps \in R^\delta \,\,\text{or}\,\, \abs{A^\eps - B^\eps} \geq \delta)  \\
&\leq 2\left(\pr(A^\eps \in R^\delta) + \pr(\abs{A^\eps - B^\eps} \geq \delta)\right).
\end{align*}
Now note that by the Markov inequality
\[
 \pr(\abs{A^\eps - B^\eps} \geq \delta) 
\leq \frac{E\left[\abs{A^\eps - B^\eps}^2\right]}{\delta^2}  
\leq \frac{C_0^2\eps^2}{\delta^2}.
\]
By assumption, we therefore get for sufficiently small $\delta$ that
\[
\abs{E(\mathds{1}_R (A^\eps)  - E(\mathds{1}_R(B^\eps)))} \leq 2\left(C_1 \delta +  C_0^2\eps^2/\delta^2\right).
\]
Choosing $\delta = \eps^{0.6}$ makes the first term $o(\sqrt{\eps})$. The second term is then proportional to $\eps^2/\delta^2 = \eps^{0.8} = o(\sqrt{\eps})$, which proves the claim.
\end{proof}

\section{Parabolic PDEs}
\label{sec:parabolic}

We collect here some useful technical facts about parabolic PDE of the form
\begin{equation}\label{eq:basicparabolic}
\partial_t u = \eps \Delta_t u\,, \quad u(0,\cdot) = u_0(\cdot)\,,
\end{equation}
where $\Delta_t u \coloneqq {\divergence_{\mass}} g_t^{-1}\dd u$ is a Laplace-like operator on a manifold $M$ for every $t \in [0,1]$, and $(g_t)_{t\in[0,1]}$ is a smoothly varying nonvanishing family of Riemannian metrics.
We write $P_t^\eps$ for the time-$t$ solution operator, i.e., $u(\cdot,t) = P_t^\eps u_0$. 
In the case that $M$ is a compact Riemannian manifold (possibly with Dirichlet boundary), we have summarized some well-known existence and uniqueness results for $u_0 \in L^2(M,\mass)$ in Appendix D of \cite{Karrasch2020a}.
If $M = \R^n$, we use in this document the assumption (\cref{assumption}) that 
there exists a bounded set $B$ (containing $S$ in its interior) so that both $g_t$ and $\mass$ are equal to the Euclidean metric and its volume form respectively outside of $B$ for all $t \in [0,1]$. 
Under this restriction, it is well-known that the time-$t$ solution operator $P_t^\eps$ is well-defined for $u_0 \in C_b(\R^n)$,
and a maximum principle for initial values in $C_0(\R^n)$ (continuous functions vanishing on infinity) holds.
The solution $u_\eps(t,x) \coloneqq P^\eps_t u_0$ satisfies \cref{eq:basicparabolic} everywhere on $(0,1] \times \R^n$ if $u_0$ has compact support.
Moreover, $P^\eps_t$ is of the form $(P^\eps_t u_0)(x) = \int_{\R^n} p_\eps(0,x,t,y)u_0(y) \mass$. We have here taken the somewhat unconventional step of using $\mass$ instead of the Lebesgue measure for the definition of the fundamental solution as this is the natural measure for problems like \cref{eq:basicparabolic},  recall that $\mass$ is equivalent to the $n$-dimensional Lebesgue measure $\ell^n$ under \cref{assumption} .
As a reference for these statements, see for instance \cite[Ch.~3--4]{Friedman1975}, \cite[Ch.~3]{Stroock2006} and \cite{Stewart1974}. 
\begin{remark}\label{remark:boundedConv}
The measure $p_\eps(0,x,t,\cdot)\mass$ is a probability measure, so we may extend $P^\eps_t$ to act on $u_0 \in L^\infty(\R^n)$. Moreoever, if $u_n \uparrow u$ pointwise everywhere for a sequence of functions $u_n \in L^\infty(\R^n)$, then the monotone convergence theorem yields that $P^\eps_t u_n \uparrow P^\eps u$.
\end{remark}
For positive initial data, the time-dependent heat equation preserves the integral with respect to $\mass$. This may be seen by adapting the proof of \cite[Sec.~6, Thm.~4.7]{Friedman1975}, but using the $L^2(\R^n,\mass)$ adjoint (as opposed to the $L^2(\R^n,\dd \ell^d)$ adjoint considered there) of $\mathcal M \coloneqq \eps \Delta_t - \partial_t$ which is given by $\mathcal M^* = \eps
\Delta_t + \partial_t$.
 The fundamental solution for $\mathcal M^*$ (adapted to $\mass$ instead of the Lebesgue measure as before), denoted by $p^*(x,t,y,\tau)$
 satisfies $\mathcal M^*p^*(\cdot,\cdot,y,\tau) = 0$ and (by mirroring the aformentioned proof) also $p^*(x,t,y,\tau) = p(y,\tau,x,t)$. As a consequence, 
\begin{align*}
 \int (P_t^\eps u_0)(x,t) \mass(x) &=  \int \int p(x,t,y,0)u_0(y) \mass(y)\mass(x) \\
&= \int \int p^*(y,0,x,t)u_0(y) \mass(y)\mass(x)  \\
&= \int \int p^*(y,0,x,t) \mass(x)u_0(y)\mass(y)  \\
&= \int u_0 \mass\,,
\end{align*}
as $p^*(y,0,\cdot,t)\mass$ is a probability measure.
Of course, all of the arguments above may also be applied to $\overline P^\eps$. In addition, here it is known that $\overline \Delta$ 
generates an analytic semigroup on $L^p(\R^n)$ for $p \in [1,\infty)$ \cite[Sect.~5.4, Thm.~5.6]{Tanabe1997},
and on $C^0(\R^n)$ \cite{Stewart1974}.

We will make use of the following approximation result.

\begin{proposition}[{\cite{Krol1991,Karrasch2020a}}]\label{thm:approximation}
If $M$ is compact (possibly with smooth homogeneous Dirichlet/Neumann boundary) and $u_0 \in C^\infty_c(\mathring{M})$ then 
\[
({P^\eps_t} u_0)(x) = u_0 + \eps\int_0^t \Delta_\tau u_0(x) \, \dd \tau + O(\eps^2) 
\]
uniformly in $(t,x)\in [0,1] \times M$ as $\eps \to 0$. 
\end{proposition}

By adapting the proof in \cite{Karrasch2020a}, this result can be extended to
the case that $M=\R^n$ assuming that the boundedness condition mentioned
earlier holds. In fact, the case $M=\R^n$ is close to the original setting of \cite{Krol1991}
on which the proof in \cite{Karrasch2020a} is based.

We conclude with the following useful property of $P^\eps_1$.

\begin{lemma}\label{lemma:kindofselfadjoint}
Let $S \subset \R^n$ be compact and measurable.  Then
\[
\Bra{P^\eps \mathds{1}_S, \mathds{1}_{S^c}}_0 = \Bra{\mathds{1}_S, P^\eps \mathds{1}_{S^c}}_0\,.
\]
Note that $P^\eps_1$ is generally not self-adjoint.
\end{lemma}

\begin{proof}
Using the properties of $P^\eps_1$ mentioned above, we compute
\begin{align*}
\Bra{P^\eps \mathds{1}_S, \mathds{1}_{S^c}}_0 &= \Bra{(P^\eps(\mathds{1}_{\R^n} - \mathds{1}_{S^c}))(\mathds{1}_{\R^n} -  \mathds{1}_{S}), \mathds{1}_{\R^n}}_0 \\
&= \Bra{(P^\eps \mathds{1}_{S^c})\mathds{1}_S + (\mathds{1}_{\R^n} - P^\eps \mathds{1}_{S^c}- \mathds{1}_S  )  ,\mathds{1}_{\R^n}}_0\\
&= \Bra{(P^\eps \mathds{1}_{S^c})\mathds{1}_S + ( P^\eps \mathds{1}_{S}- \mathds{1}_S  )  ,\mathds{1}_{\R^n}}_0\\
&= \Bra{P^\eps \mathds{1}_{S^c}, \mathds{1}_S}_0 + 0\,.\qedhere
\end{align*}
\end{proof}


\end{document}